\documentclass[10pt,a4paper]{amsart}
\pdfoutput=1 
\usepackage[
text={440pt,575pt},
headheight=9pt,
centering
]{geometry}
\usepackage[foot]{amsaddr}

\usepackage{amsmath}
\usepackage{amsfonts}
\usepackage{amssymb}
\usepackage{amsthm}
\usepackage{graphicx}
\usepackage{url}
\usepackage{natbib}
\usepackage{xcolor}
\usepackage{hyperref}
\usepackage{url, hypcap}
\definecolor{darkblue}{RGB}{0,0,160}
\hypersetup{
	colorlinks,%
	citecolor=darkblue,%
	filecolor=black,%
	linkcolor=darkblue,%
	urlcolor=darkblue
}

\theoremstyle{plain}
\newtheorem{thm}{Theorem}

\newtheorem{cor}[thm]{Corollary}
\newtheorem{prop}[thm]{Proposition}

\theoremstyle{remark}
\newtheorem{example}[thm]{Example}
\newtheorem{remark}[thm]{Remark}
\newtheorem{defn}[thm]{Definition}

\newtheorem{algorithm}{Algorithm}

\newcommand{\x}{\mathbf{x}}
\newcommand{\y}{\mathbf{y}}
\newcommand{\z}{\mathbf{z}}

\newcommand{\X}{\mathbf{X}}
\newcommand{\Y}{\mathbf{Y}}
\newcommand{\W}{\mathbf{W}}
\newcommand{\Z}{\mathbf{Z}}
\newcommand{\RR}{\mathbb{R}}

\newcommand{\PP}{\mathbb{P}}
\newcommand{\EE}{\mathbb{E}}

\newcommand{\Cov}{\text{Cov}}
\newcommand{\MTPtwo}{$ \text{MTP}_2 $ } 
\newcommand{\MTPtwok}{$ \text{MTP}_2 $, }
\newcommand{\MTPtwop}{$ \text{MTP}_2 $. }
\newcommand{\EMTPtwo}{$ \text{EMTP}_2 $ }
\newcommand{\EMTPtwok}{$ \text{EMTP}_2 $, }

\newcommand{\indep}{\perp \!\!\! \perp}
 
\newcommand{\bs}{\boldsymbol}
\newcommand{\CM}{\operatorname{CM}}
\newcommand{\A}{\operatorname{A}}

\DeclareMathOperator*{\argmax}{\operatorname{arg max}}
\DeclareMathOperator*{\argmin}{\operatorname{arg min}}

\DeclareFontFamily{U}{mathx}{\hyphenchar\font45}
\DeclareFontShape{U}{mathx}{m}{n}{
	<5> <6> <7> <8> <9> <10>
	<10.95> <12> <14.4> <17.28> <20.74> <24.88>
	mathx10
}{}
\DeclareSymbolFont{mathx}{U}{mathx}{m}{n}
\DeclareFontSubstitution{U}{mathx}{m}{n}
\DeclareMathAccent{\widecheck}{0}{mathx}{"71}

\begin{document}
\title{On the local metric property in multivariate extremes}
\author[F.~R\"ottger]{Frank R\"ottger$^1$}
\author[Q.~Schmitz]{Quentin Schmitz$^1$}
\address[]{$^1$Research Center for Statistics\\Université de Genève\\ 1205 Geneva\\Switzerland}
\email{frank.roettger@unige.ch, quentin.schmitz@etu.unige.ch}
\date{\today}

\makeatletter
\@namedef{subjclassname@2020}{\textup{2020} Mathematics Subject Classification}
\makeatother

\begin{abstract}
Many multivariate data sets exhibit a form of positive dependence, which can either appear globally between all variables or only locally within particular subgroups.
A popular notion of positive dependence that allows for localized positivity is positive association.
In this work we introduce the notion of extremal positive association for multivariate extremes from threshold exceedances. 
Via a sufficient condition for extremal association, we show that extremal association generalizes extremal tree models. 
For H\"usler--Reiss distributions the sufficient condition permits a parametric description that we call the metric property. 
As the parameter of a H\"usler--Reiss distribution is a Euclidean distance matrix, the metric property relates to research in electrical network theory and Euclidean geometry.
We show that the metric property can be localized with respect to a graph and study surrogate likelihood inference. This gives rise to a two-step estimation procedure for locally metrical H\"usler--Reiss graphical models.
The second step allows for a simple dual problem, which is implemented via a gradient descent algorithm. 
Finally, we demonstrate our results on simulated and real data.
\end{abstract}
\maketitle
\section{Introduction}

Notions of positive dependence are important in many applications in statistics, as they have desirable properties in asymptotics or computations, permit simple interpretation in terms of conditional independence or correlation and appear in many real-world examples.
Applications include financial data \citep{WRU2020}, gene networks and phylogenetics \citep{SUZ2020}, sequential testing \citep{Benjamini2001} or multivariate extremes \citep{REZ2021}.
In multivariate extremes, positive dependence is observed frequently, for example because extreme events in financial or meteorological data are often driven by common factors.
Such observations are supported by asymptotic theory, as for example max-stable distributions are always positively associated \citep{MO1983}.

Positive association is a popular notion of positive dependence first studied by \citet{Esary1967}, see also \cite{LZ2022}.
In this paper we introduce a notion of extremal positive association for multivariate Pareto distributions, which are the only possible limit distributions for threshold exceedances \citep{Rootzen2006}.
Our work connects to the new framework of extremal conditional independence and extremal graphical models for multivariate Pareto distributions \citep{EH2020}.
This work has created a new line of research in graphical extremes, for example in structure learning \citep{EV2022,ELV2021,LO2023} or in parameter estimation \citep{HES2022,RCG2023}.
Furthermore, \citet{REZ2021} define an extremal version of a strong notion of positive dependence called multivariate total positivity of order two ($\text{MTP}_2$). 

While \MTPtwo has many interesting properties for statistical inference \citep{LUZ2019}, it is by construction a global property that in general can not be encoded by any collection of marginals. 
For example, a multivariate Gaussian is \MTPtwo when all partial correlations are non-negative.
In contrast, a multivariate Gaussian is positively associated if and only if all correlations are non-negative \citep{Pitt1982}.
This means that a multivariate Gaussian is positively associated when all its bivariate marginals are associated.
\citet{LZ2022} employ this structure for a natural localization of positive association, where for a given undirected graph $ G=(V,\mathcal{E}) $ with vertex set $ V $ and edge set $ \mathcal{E}\subset V\times V $ all marginals indexed by the cliques of $ G $ are supposed to be positively associated.
For Gaussians, this simplifies such that each bivariate marginal corresponding to an edge in $ G $ is required to have non-negative correlation.
This allows for more flexible models that can incorporate positive dependence locally, as it does not impose potentially unrealistic non-negative correlations between disconnected vertices in the graph $ G $.
\citet{LZ2022} discuss estimation for graphical models under local association and introduce the mixed dual estimator for this problem, which provides a general estimation procedure under convex constraints for mixed convex exponential families.

In the positive dependence literature, a recurring example for models with intrinsic positive dependence are latent variable models.
For example, \citet{Karlin1980} study \MTPtwo and \citet{HR1986} the notions of conditional positive association and latent \MTPtwo for unidimensional latent variable distributions, i.e.~latent variable distributions with a single latent variable.
Furthermore, \citet{Benjamini2001} show that in multiple testing problems the false discovery rate can be controlled under the assumption that the vector of test statistics is a latent variable model that satisfies a weaker form of positive dependence called positive regression dependency \citep{Lehmann1966,sarkar1969}.

The first contribution of this work is based on the conditional latent variable structure of multivariate Pareto distributions.
After introducing a notion of extremal association for multivariate Pareto distributions (Section~\ref{sec:extr_assoc}), we find a sufficient condition for extremal association in Theorem~\ref{t:extr_assoc}. 
This follows from the observation (in Proposition~\ref{prop:latent_factor_assoc}) that given a univariate random variable $ X_0 $ which is independent of a positively associated random vector $ \X $, the linear latent variable model $ X_0\mathbf{1}+\X $ is also positively associated.
As a consequence, we find that bivariate Pareto distributions and more generally extremal tree models are always extremal associated (see Corollary~\ref{cor:trees}).

The second contribution of this work is a detailed investigation of extremal association for H\"usler--Reiss distributions, which we summarize below. 

\subsection{Extremal association for H\"usler--Reiss distributions}

H\"usler--Reiss distributions are a parametric family of multivariate Pareto distributions that can be considered as an extremal analogue of Gaussians. They are parameterized by a variogram matrix $ \Gamma $, that is a Euclidean distance matrix. 
\cite{HES2022} introduce a H\"usler--Reiss precision matrix $ \Theta $ that is one-to-one to $ \Gamma $ and show that zeros in $ \Theta $ encode extremal conditional independence.
This allows for a simple parametric description of H\"usler--Reiss graphical models, so that for a given undirected graph $ G=(V,\mathcal{E}) $ a non-edge $ ij\not\in\mathcal{E} $ implies a vanishing entry $ \Theta_{ij}=0 $ in the precision matrix.
\cite{REZ2021} prove that a H\"usler--Reiss distribution is extremal \MTPtwo ($\text{EMTP}_2$) when $ \Theta $ is the Laplacian matrix of an undirected graph.

In Section~\ref{sec:loc_metr} we find that the trivariate constraints
\[\Gamma_{ij}\le \Gamma_{ik}+\Gamma_{jk},\;\text{ for all }\;  i,j,k\in \{1,\ldots,d\} \]
are a sufficient condition for $ d $-variate H\"usler--Reiss distributions to be extremal associated.
In this case we say that $ \Gamma $ satisfies the \textit{metric property} as $ \Gamma $ is by definition non-negative and symmetric.
If $ \Gamma $ is interpreted as a Euclidean distance matrix, the metric property means that every triangle with squared edge lengths $ (\Gamma_{ij},\Gamma_{ik},\Gamma_{jk}) $ is hyperacute.
When $ \Theta $ is a Laplacian matrix, then the corresponding $ \Gamma $ always fulfills the triangle inequality, while the inverse statement is not true \citep{Devriendt2022}.
This observation has an interesting parallel in electrical network theory. 
When the matrix of resistances $ \Theta $ of the network is a Laplacian matrix, i.e., all resistances are positive, the effective resistance matrix $ \Gamma $ is guaranteed to define a distance function.
In practice there might be negative resistances in a network \citep{Chen2016}, so that the metric interpretation is not guaranteed.
In fact, for large electrical networks a small proportion of triangle inequalities were observed to be negative \citep{Lagonotte1989,Cotilla2012}.

Trivariate H\"usler--Reiss distributions are \EMTPtwo if and only if they comply with the metric property \cite[Example~5]{REZ2021}, which resembles the equivalence of \MTPtwo and positive association for bivariate Gaussians.
As the metric property is trivariate, this permits a simple localization such that the \textit{local metric property} holds with respect to a graph $ G=(V,\mathcal{E}) $ when for each triangle in $ G $ the corresponding marginal satisfies the metric property.
If in addition a H\"usler--Reiss distribution is an extremal graphical model we speak of a \textit{locally metrical graphical model}.

In Section~\ref{sec:estimation} we develop statistical methodology for locally metrical H\"usler--Reiss graphical models.
We adapt the two-step estimation procedure of \cite{LZ2022} and find the dual of the second step in Proposition~\ref{prop:dual}.
Our new methodology provides an estimate for $ \Gamma $ that is Markov and satisfies the metric property with respect to $ G $.
The dual for the second step gives a simple tool to minimize a surrogate reciprocal Kullback--Leibler divergence between given graph weights and the set of locally metrical models.
For electrical networks, this finds the closest resistance distance matrix that allows the interpretation as a metric.

The particularly simple structure of the dual problem enables us to implement a gradient descent algorithm for the second step in Section~\ref{sec:algorithm}.
For this, we employ the method of moving asymptotes (MMA) \citep{Svanberg2001} implemented in the \texttt{R} package \texttt{nloptr} \citep{nlopt}.
We demonstrate the performance of this algorithm on simulated data.

In Section~\ref{sec:application} we study the flight delays data set of \citet{HES2022} in the context of locally metrical H\"usler--Reiss graphical models.
We select $ d=79 $ airports with more than 2000 annual flights between 2005 and 2020 and ignore any day with missing data. This results in a data set with $ n=5347 $ daily observations of (positive) total delays.
Following \citet{HES2022}, we use the \texttt{eglearn} structure learning algorithm of \cite{ELV2021} on a training data set to estimate a sparse H\"usler--Reiss graphical model.
We compute the first step estimates and compare their H\"usler--Reiss log-likelihoods on a validation data set.
While the first step estimates are not \EMTPtwok we find that they satisfy the local metric property with respect to the estimated graph for sufficiently large regularization parameters, including the first step estimate with the highest log-likelihood on the validation data. 
This gives credit to the intuition that the local metric property is a very reasonable assumption for data sets with intrinsic local positive dependence.

The \texttt{R}-code for the simulation and application is available as a GitHub repository: \url{https://github.com/frank-unige/on_the_local_metric_property_in_multivariate_extremes}

\section{Preliminaries}
\subsection{Threshold exceedances and extremal conditional independence}\label{sec:preliminaries1}
In multivariate extremes, our interest is in the tail properties of a random vector $ \X=(X_1,\ldots,X_d) $.
When concern is in extremal dependence structures, it is common to assume that $ \X $ is multivariate regularly varying \citep{Resnick2008} and normalized to standard exponential margins \citep{HES2022}.
The distribution of the exceedances of $ \X $ over a high threshold converges to a \textit{multivariate (generalized) Pareto distribution} \citep{Rootzen2006}, that is
\begin{align}
	\PP(\Y\le \y)=\lim_{u\to\infty}\PP(\X-u\mathbf{1}\le \y | \|\X\|_{\infty}>u)\label{eq:MPD}
\end{align}
for $ \y\in\mathcal{L}:=\{\x\in\RR^d:\|\x\|_{\infty}>0\} $ where $ \mathbf{1}:=(1,\ldots,1)^T $ is the vector of ones. 
Note that $ \mathcal{L} $ is not a product space, and that the vector $ \Y $ satisfies a homogeneity property $ \PP(\Y\in t\mathbf{1}+A)=t^{-1}\PP(\Y\in A) $ for any $ t>0 $ and Borel set $ A\subset \mathcal{L} $.
Let $ \X_I $ be the $ I $th marginal of $ \X $ for a non empty index set $ I\subset [d] $.
We denote as $ \Y_I $ the limit in \eqref{eq:MPD} for $ \X_I $.
This is a slight abuse of notation for convenience as $ \Y_I $ is not the $ I $th marginal of $ \Y $, but the $ I $th marginal conditioned on an exceedance for at least one $ k\in I $, see \citet{HES2022}.
Following \citet{EH2020}, let
$\Y^k:=\Y|Y_k>0$ for $ k\in [d]:=\{1,\ldots,d\}$
be conditional random vectors with support $ \mathcal{L}^k:=\{\y\in\mathcal{L}:y_k>0\} $.
Note that the union of $ \mathcal{L}^k,\; k\in[d] $ equals $ \mathcal{L} $.
As $ \mathcal{L}^k $ are product spaces, this allows for an extremal notion of conditional independence $ \perp_e $ via
\[\Y_A\perp_e \Y_B|\Y_C\Longleftrightarrow \;\forall k\in [d]:\; \Y^k_A\indep \Y_B^k|\Y_C^k \]
for disjoint sets $ A,B,C\subset [d] $.
Extremal conditional independence is therefore equivalent to simultaneous stochastic conditional independence statements on each half-space $ \mathcal{L}^k $.
The notion of extremal conditional independence gives rise to an extremal graphical model with respect to an undirected graph $ G=(V,\mathcal{E}) $ with vertex set $ V=[d] $ and edge set $ \mathcal{E} $ by
$Y_i\perp_eY_j|\Y_{V\setminus ij}$ for all $ ij\not\in \mathcal{E},$
i.e.~for every non-edge in $ G $ we require extremal conditional independence between $ Y_i $ and $ Y_j $ given the remaining variables $ \Y_{V\setminus ij} $.
For convenience, we abbreviate index set differences of the form $ V\setminus ij $ to $ \setminus ij $ whenever the minuend set is clear from context.
By definition extremal graphical models do not allow marginal independence, so that the graphs are always connected.
The recent work of \citet{EIS2022} shows how to generalize the definition of extremal conditional independence to incorporate disconnected graph.

A useful property of the conditional random vectors $ \Y^k $ is that they allow a stochastic representation as a latent variable model
\begin{align}
	\Y^k\stackrel{d}{=}E\mathbf{1}+\W^k, \label{eq:extr_function}
\end{align}
where $ E $ is a standard exponential random variable that is independent of a $ d $-variate random vector~$ \W^k $.
It holds that $ W_k^k=0 $ almost surely and we refer to $ \W^k $ as the $ k $th extremal function.
The $ k $th extremal function uniquely characterizes a multivariate Pareto distribution, and for $ k\in C $ there is a simple equivalence such that
$\Y_A\perp_e \Y_B|\Y_C$ if and only if $ \W_A^k\indep\W_B^k|\W_C^k.$
Furthermore, extremal functions allow for a natural construction of particular multivariate Pareto distributions like the H\"usler--Reiss distribution, which we will introduce below.

\subsection{H\"usler--Reiss distributions}\label{s:hr} 
The H\"usler--Reiss distribution is a subclass of multivariate Pareto distributions that is considered as the analogue of the multivariate Gaussian in multivariate extremes, see e.g.~\citet{HR1989} or \citet{EMKS2015}.
An elegant construction of the $ d $-variate H\"usler--Reiss distribution is via the $ k $th extremal function $ \W^k $,
where we require $ \W^k_{\setminus k} $ to follow a $ (d-1) $-variate Gaussian with covariance $ \Sigma^{(k)} $ and mean vector $ -\text{diag}(\Sigma^{(k)})/2 $.
A parameterization that is independent of $ k $ is obtained via the inverse covariance mapping
\[\begin{cases}
	\Gamma_{ik}=\Sigma_{ii}^{(k)}, & i\neq k,\\
	\Gamma_{ij}=\Sigma_{ii}^{(k)}+\Sigma_{jj}^{(k)}-2\Sigma_{ij}^{(k)}& i,j\neq k,
\end{cases}\]
with $ \text{diag}(\Gamma)=\mathbf{0} $. The covariance map $ \Sigma_{ij}^{(k)}= (\Gamma_{ik}+\Gamma_{jk}-\Gamma_{ij})/2,\; i,j\in[d] $ that returns the parameter matrix $ \Sigma^{(k)} $ for any $ k\in [d] $ is known as \textit{Farris transform} in phylogenetics~\citep{SUZ2020}.
Let $ \mathcal{H}^{d-1}:=\{\x\in\RR^d:\x^T\mathbf{1}=0\} $ denote the hyperplane perpendicular to $ \mathbf{1} $.
The matrix $ \Sigma^{(k)} $ is positive definite if and only if $ \Gamma $ is strictly conditionally negative definite, i.e.~$ \x^T\Gamma\x< 0 $ for all $ \x \in \mathcal{H}^{d-1}\setminus\{\mathbf{0}\} $.
We define $ \mathcal{C}^d $ as the cone of conditionally negative definite $ d\times d $ matrices.
The matrix $ \Gamma $ is a variogram matrix that can be obtained from $ \Y $ via
\[\Gamma_{ij}=\text{Var}(Y^k_i-Y^k_j)\]
for each $ k\in[d] $.
Furthermore, when $ \Gamma $ is conditionally negative definite it is a Euclidean distance matrix, i.e.~the entries in $ \Gamma $ allow a representation as the squared distances of $ d $ points in Euclidean space, such that the matrix $ (\sqrt{\Gamma_{ij}})_{i,j\in[d]} $ defines a metric as $ \sqrt{\Gamma_{ij}}\le \sqrt{\Gamma_{ik}}+\sqrt{\Gamma_{jk}} $ for all $ i,j,k\in [d] $.

H\"usler--Reiss distributions allow a simple parametric characterization of extremal conditional independence.
In fact, let $ \Theta^{(k)}:=\left(\Sigma^{(k)}\right)^{-1} $ for each $ k\in[d] $ and define a positive semidefinite $ d\times d $-matrix $ \Theta $ so that $ \Theta_{ij}=\Theta_{ij}^{(k)} $ for $ i,j\neq k $.
It follows that
\begin{align}
	Y_i\perp_e Y_j|\Y_{\setminus ij} \Longleftrightarrow \Theta_{ij}=0. \label{eq:HRprecision}
\end{align}
We refer to $ \Theta $ as the \textit{H\"usler--Reiss precision matrix} \citep{HES2022}. 
The equivalence in \eqref{eq:HRprecision} implies that a H\"usler--Reiss graphical model equivalently imposes zeros in its precision matrix, such that problems with respect to sparsity allow a parametric reformulation.
Let $ \mathbb{S}^d_0 $ denote the set of symmetric $ d\times d $-matrices with zero diagonal and let $Q\in \mathbb{S}^d_0$ be defined by $Q_{ij}:=-\Theta_{ij}$ for all $i\neq j \in [d]$.
As $ \Theta\mathbf{1}=0 $, it follows that $ Q $ and $ \Theta $ are one-to-one with $ \Theta_{ii}=\sum_{j\neq i} Q_{ij} $.
In the context of H\"usler--Reiss graphical models the matrix $ Q $ assigns graph weights to each edge in the underlying graph and equals zero for non-edges, so that $ Q $ is a weighted adjacency matrix.
A simple H\"usler--Reiss graphical model is given in Example~\ref{ex:HR_small}.

Let $ P:=I-\frac{1}{d}\mathbf{1}\mathbf{1}^T $ and define $ \Sigma:=P(-\Gamma/2)P $.
It holds that $ \Sigma=\Theta^+ $ is the pseudo-inverse of $ \Theta $, compare also \citet{HES2022} or \cite{REZ2021}.
The matrix $ \Sigma $ is the degenerate Gaussian covariance in a stochastic representation of the conditional vector $ \widetilde{\Y}:=\Y|\{\Y^T\mathbf{1}>0\} $, that is
\begin{align}
	\widetilde{\Y}\stackrel{d}{=}E\mathbf{1}+\W \label{eq:representation_Hentschel}
\end{align}
with $ E $ standard exponential and independent of a degenerate Gaussian random vector $ \W $ with mean $ P(-\frac{\Gamma}{2d})\mathbf{1} $ and covariance $ \Sigma $ that is supported on the hyperplane $ \mathcal{H}^{d-1} $. 
We will refer to the random vector $ \W $ as \textit{extremal function}. 

\subsection{Positive association}
We call a function $f:\RR^d\to \RR$ increasing when $x_i\le x_i'$ for all $1\le i\le d$ implies that $f(\x)\le f(\x')$.
A random vector $ \X $ is \textit{positively associated} when for all increasing $ f,g $ it holds that
$$
\Cov(f(\X),g(\X))\ge 0, 
$$
whenever the covariance exists.
A multivariate Gaussian with covariance matrix $\Sigma$ is positively associated when $\Sigma_{ij}\ge 0$ for all $ i,j\in[d] $ \citep{Pitt1982}.
For further details on positive association, we refer to \citet{Esary1967} and \citet{LZ2022}.
For notational simplicity, we will denote positive association simply as \textit{association}.

\citet{LZ2022} introduced a localized version of association. 
It requires that fully connected subgroups of variables with respect to some graph are positively associated. 
For convenience, we note some basic concepts of undirected graph theory and graphical models. We refer to \citet{Lauritzen1996} for a full introduction.
We consider only \textit{undirected} graphs $G = (V, \mathcal{E})$ with vertex set $ V=[d] $, i.e.~graphs where for every edge $(i,j) \in \mathcal{E}$ we have also $(j,i) \in \mathcal{E}$. 
We call \textit{complete} a graph where all vertices are connected by an edge.
A subgraph $G_S = (V_S, \mathcal{E}_S)$ of a graph $G$ is the graph $G$ where we only keep the vertices $V_S \subset V$ and the corresponding edges $\mathcal{E}_S=\{ij\in \mathcal{E}: i,j\in V_S\}$.  
We call a complete subgraph $C=(V_C,\mathcal{E}_C)$ a \textit{clique} when it is maximal with respect to inclusion, which means that we cannot add another vertex $k \in V_{\setminus_C}$ so that the subgraph with vertex set $C \cup k$ is still complete. 
We denote by $\mathcal{C}(G)$ the set of cliques of $G$.
The definition of local association is as follows:
\begin{defn}\label{def:locasso}
	\emph{A random vector ${\bf X}$ is \textit{locally associated} with respect to a graph $G$ if the subvector $\X_C$ is associated for any clique $C\in\mathcal{C}(G)$.}
\end{defn}
Local association is a more flexible notion of positive dependence as it incorporates a graph structure. Only lower dimensional marginals corresponding to cliques are required to be associated.
As marginals of Gaussians are distributed according to a Gaussian with the corresponding submatrix as covariance, local association allows a simple parametric description for Gaussians:
\begin{example}
	A multivariate Gaussian with covariance matrix $\Sigma$ is locally associated with respect to a graph $ G $ if and only if $\Sigma_C \geq 0, \forall C \in \mathcal{C}(G)$.
\end{example}

\section{Extremal association and the local metric property}

In this section, we introduce the notion of extremal association and find a sufficient condition for this property to hold. 
Afterwards, we discuss a parametric description of the sufficient condition for H\"usler--Reiss distributions. 

\subsection{Latent variable models and extremal association}\label{sec:extr_assoc}

As discussed in Section~\ref{sec:preliminaries1}, the conditional vectors $ \Y^k:=\Y|\{Y_k>0\} $ that characterize extremal conditional independence allow a stochastic representation as a latent variable model $ E\mathbf{1}+\W^k $, where $ E $ is a standard exponential random variable that is independent of the $ k $th extremal function $ \W^k $. 
Henceforth, properties of latent variable models provide insights for multivariate Pareto distributions and assumptions on the extremal functions $ \W^k,\;k\in[d] $ will impose structure in the resulting extremal model. 

In this section, we link positive dependence properties for latent variable models with corresponding extremal properties for multivariate Pareto distributions.
Let $ \X $ be a random vector independent of a univariate random variable $ X_0 $. We call
\begin{align}
	X_0\mathbf{1}+\X \label{eq:latent_factor}
\end{align}
a \textit{latent variable model}.
Previous research investigated how notions of positive dependence for $ \X $ give rise to positive dependence for the latent variable model.
For example, \citet{Benjamini2001} showed for \textit{positive regression dependency on a subset} (PRDS) that $ \X $ being PRDS implies that $ X_0\mathbf{1}+\X $ is PRDS.
A similar result involving strongly \MTPtwo distributions can be obtained for \MTPtwo latent variable models, see \citet[Theorem~3.1]{REZ2021}.
The following result shows that a sufficient condition for latent variable models \eqref{eq:latent_factor} to be  associated is when $ \X $ is associated:
\begin{prop}\label{prop:latent_factor_assoc}
	Let $ \X $ be an associated random vector and $ X_0 $ an independent univariate random variable. Then $ 	X_0\mathbf{1}+\X $ is associated.	
\end{prop}
\begin{proof}
	As univariate random variables are always associated and compositions of independent associated random vectors are associated \citep{Esary1967}, $ (X_0,\X) $ is associated.
	Let $ \Z:=(X_0,X_0\mathbf{1}+\X) $. Any non-decreasing function $ f(\Z)=f(X_0,X_0\mathbf{1}+\X) $ corresponds to a non-decreasing function $ \tilde{f}(X_0,\X) $ with $ \tilde{f}(x_0,\x):=f(x_0,x_0\mathbf{1}+\x) $.
	Therefore $ \Z $ is associated. As association is closed under marginalization \citep{Esary1967}, the result follows.
\end{proof}

The conditional vectors $ \Y^k $ are supported on product spaces $ \mathcal{L}^{k} $, which allows for the definition of extremal conditional independence via stochastic conditional independence for all $ \Y^k,\; k\in[d] $.
In a similar fashion, \citet{REZ2021} introduce a notion of extremal \MTPtwok so that $ \Y $ is extremal \MTPtwo ($ \text{EMTP}_2 $) if and only if $ \Y^k $ is \MTPtwo for all $ k\in[d] $.
This approach leads to a definition of an extremal notion of association:
\begin{defn}
	We call a multivariate Pareto vector $ \Y $ extremal associated if and only if $ \Y^k $ is associated for each $ k\in V $.
\end{defn}
We observe that the set of extremal associated distributions contains the set of \EMTPtwo distributions.
This directly implies that many parametric families like the extremal logistic and Dirichlet distributions satisfy extremal association for every parameter \citep{REZ2021}.

Through the representation \eqref{eq:extr_function} of $ \Y^k $ as a latent variable model $ E\mathbf{1}+\W^k $, \EMTPtwo is characterized via the extremal functions, so that $ \Y $ is \EMTPtwo if and only if $ \W^k_{\setminus k} $ is \MTPtwo for each $ k\in V $.
Via Proposition~\ref{prop:latent_factor_assoc} we obtain that $ \W^k $ being associated is a sufficient condition to imply extremal association.
\begin{thm}\label{t:extr_assoc}
	When $ \W^k $ is associated for each $ k\in [d] $ it follows that $ \Y $ is extremal associated. 
\end{thm}
\begin{proof}
	The theorem follows from Proposition~\ref{prop:latent_factor_assoc} and the definition of extremal association.
\end{proof}
The result relies on the latent variable structure of $ \Y^k $.
As a corollary of Theorem~\ref{t:extr_assoc}, we find that extremal tree models, i.e.~multivariate Pareto distributions that are Markov to a tree, are always extremal associated.
This further shows that bivariate Pareto distributions are always extremal associated.
\begin{cor}\label{cor:trees}
	Extremal tree models are extremal associated.
\end{cor}
\begin{proof}
	Let $ \Y $ be a $ d $-variate multivariate Pareto distribution that is Markov to a tree $ G=(V,\mathcal{E}) $.
	According to \cite{EV2022}, in this case for each $ k\in V $ the extremal function allow a representation via $ d-1 $ independent univariate random variables $ (W_e)_{e\in\mathcal{E}} $, so that
	\[W_{i}^{k}=\sum_{e\in\text{ph}(ik)}W_e, \]
	where $ \text{ph}(ik) $ is the unique path between $ i $ and $ k $.
	As univariate random variables are associated and the concatenation of independent associated random vectors is associated \citep{Esary1967}, the vector $ (W_e)_{e\in\mathcal{E}} $ is associated.
	Finally, as $ \W^k $ is a nondecreasing function of $ (W_e)_{e\in\mathcal{E}} $, we obtain that $ \W^k $ is associated, which gives the result.
\end{proof}

\subsection{Local positive dependence for H\"usler--Reiss distributions}\label{sec:loc_metr}
We will now specify the results on extremal association for the important subclass of H\"usler--Reiss distributions, which are considered to be the analogue of Gaussians among multivariate Pareto distributions.

For multivariate Gaussian distributions, association is a bivariate property that is encoded in the covariance matrix, as it is equivalent to non-negative correlations between all variables.
In fact, this means that a multivariate Gaussian $ \X $ is associated if and only if every bivariate marginal $ \X_I,\;|I|=2 $ is associated.
Furthermore, for bivariate Gaussians the notions of association and \MTPtwo are equivalent, as a $ 2\times 2 $ precision matrix is an M-matrix if and only if the corresponding covariance matrix is non-negative.
Therefore a multivariate Gaussian is associated if and only if all its bivariate marginals are \MTPtwop

Theorem~\ref{t:extr_assoc} and Corollary~\ref{cor:trees} show that any bivariate Pareto distribution is extremal associated.
This stems from the fact that a bivariate multivariate Pareto distribution is characterized by a univariate random variable via the extremal function.
For H\"usler--Reiss distributions, any bivariate marginal is even \EMTPtwo \citep{REZ2021}.
Therefore we do not expect that there exists an interesting bivariate notion of positive dependence for H\"usler--Reiss distributions.

As a $ d $-variate Pareto vector $ \Y $ is constructed from a $ (d-1) $-variate extremal function $ \W^k_{\setminus k} $, a bivariate notion of positive dependence for $ \W_{\setminus k}^k $ should induce a trivariate property for $ \Y $.
Indeed we observe such a structure for trivariate H\"usler--Reiss distributions, as for those the extremal functions are bivariate Gaussians.
\citet[Prop.~3.6]{REZ2021} showed that $ \W^k $ is associated for all $ k\in V $
if and only if the variogram matrix $ \Gamma $ is a metric, i.e.~when the triangle inequality 
\begin{equation}
	\Gamma_{ij} \leq \Gamma_{ik} + \Gamma_{jk}\label{eq:triangle_ineq}
\end{equation}
holds for all triples $ i,j,k \in V $. We say in this case that $ \Y $ satisfies the \textit{metric property} and note that the metric property implies extremal association due to Theorem~\ref{t:extr_assoc}.
For H\"usler--Reiss distributions the vector $\Y_I$ is again H\"usler--Reiss with parameter $ \Gamma_{I} $, so that the metric property is in fact a trivariate property.
This holds as $ \Gamma $ satisfies the triangle inequalities if and only if $ \Gamma_{I} $ satisfies the triangle inequalities for every $ I\subset [d] $ with $ |I|=3 $.
Note that the metric property is defined by simple linear constraints for the parameter matrix $ \Gamma $ that enforce a natural structure for Euclidean distance matrices \citep{Devriendt2022}. 
We further observe that for trivariate H\"usler--Reiss distributions the metric property is equivalent to \EMTPtwo \cite[Example~5]{REZ2021},
mirroring the link between association and \MTPtwo for bivariate Gaussians.
As the metric property is characterized through the trivariate submatrices of $ \Gamma $, this allows for a localized notion of extremal positive dependence.
\begin{defn}\label{def:extremlocmetrHRGM}
	A H\"usler--Reiss random vector ${\bf Y}$ satisfies the \textit{local metric property} with respect to the graph $G=(V,\mathcal{E})$ if the parameter matrix $\Gamma$ is such that each of its sub-matrices corresponding to the cliques of $G$ forms a metric, i.e.
	\begin{equation}\label{locmetrGamma}
		\Gamma_{ij} \leq \Gamma_{ik} + \Gamma_{jk},\quad \forall i,j,k \in C, \forall C \in \mathcal{C}(G).
	\end{equation}
	If furthermore it holds that
	\begin{equation}\label{markov}
		\Theta_{ij} = 0, \quad \forall (i,j) \notin \mathcal{E},
	\end{equation}
	where $\Theta$ is the H\"usler--Reiss precision matrix (Section~\ref{s:hr}), i.e.~when $ \Y $ is an extremal graphical model with respect to $ G $, then we call $ \Y $ a \textit{locally metrical H\"usler--Reiss graphical model}.
\end{defn}
The following example gives a locally metrical graphical model with respect to the graph in Figure~\ref{fig:graph_HR}.
\begin{example}\label{ex:HR_small}
	Consider the H\"usler--Reiss distribution that is Markov to the graph in Figure~\ref{fig:graph_HR}, which requires that its precision matrix permits a representation
	\begin{align*}
		\Theta=\begin{pmatrix}
			Q_{12}+Q_{13} & -Q_{12} & -Q_{13}& 0\\
			-Q_{12}& Q_{12}+Q_{23}+Q_{24}& -Q_{23}& -Q_{24}\\
			-Q_{13}& -Q_{23}& Q_{13}+Q_{23}+Q_{34}& -Q_{34}\\
			0& -Q_{24}& -Q_{34}& Q_{24}+Q_{34}\\
		\end{pmatrix},
	\end{align*}
	where $ Q\in\mathbb{S}_0^d $ is a weighted adjacency matrix of the graph.
	The graph contains two cliques $ \{1,2,3\} $ and $ \{2,3,4\} $, so that the local metric property holds if and only if the variogram matrix $ \Gamma $ that corresponds to $ Q $ satisfies
	\begin{align*}
		&\Gamma_{12} \le \Gamma_{13} +\Gamma_{23}, \qquad  \Gamma_{23} \le \Gamma_{24} +\Gamma_{34},\\
		&\Gamma_{13} \le \Gamma_{12} +\Gamma_{23}, \qquad  \Gamma_{34} \le \Gamma_{23} +\Gamma_{24},\\
		&\Gamma_{23} \le \Gamma_{12} +\Gamma_{13}, \qquad  \Gamma_{24} \le \Gamma_{23} +\Gamma_{34}.	
	\end{align*}
	In this case the matrix $\Gamma$ parameterizes a locally metrical H\"usler--Reiss graphical model with respect to the graph in Figure~\ref{fig:graph_HR}.
\end{example}
\begin{figure}[h!]
	\centering
	\includegraphics[trim=0mm 0mm 0mm 7mm]{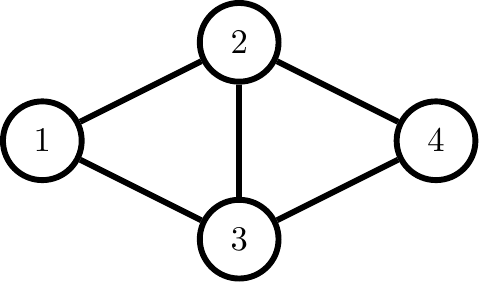}
	\caption{An undirected graph which encodes the extremal dependence structure of the H\"usler--Reiss graphical model in Example \ref{ex:HR_small}.}
	\label{fig:graph_HR}
\end{figure}

\section{Estimation for H\"usler--Reiss distributions under the local metric property}\label{sec:estimation}

In this section we develop an estimation procedure for locally metrical H\"usler--Reiss graphical models by adapting the mixed dual estimator of \citet{LZ2022} to our setting.
We begin by introducing a surrogate log-likelihood for H\"usler--Reiss distributions, which will further allow for the definition of a surrogate Kullback--Leibler divergence.

\subsection{Surrogate log-likelihood and Kullback--Leibler divergence}\label{sec:surrogate}
We study likelihood inference for H\"usler--Reiss distributions.
Following the approach in \citet{REZ2021}, we replace the H\"usler--Reiss log-likelihood with a surrogate log-likelihood in order to reduce technical difficulties.
This results in a more tractable optimization problem, compare also \citet[Section S.4.1]{HES2022}.
Let $ \Y $ be a $ d $-variate H\"usler--Reiss random vector with parameter matrix $ \Gamma $ and H\"usler--Reiss precision matrix $ \Theta $.
Consider the density of the extremal function $ \W $ in the representation $ \eqref{eq:representation_Hentschel} $, that is a degenerate Gaussian with covariance $ \Theta^+ $ and mean $ P(-\frac{\Gamma}{2d})\mathbf{1} $.
This is a degenerate $ d $-variate Gaussian distribution on $ \mathcal{H}^{d-1} $ with density
\begin{align}
	f(\x) = (2\pi)^{-d/2}\mathrm{Det}(\Theta)^{1/2} \exp\left(-\frac{1}{2} \x^T\Theta \x\right), \label{eq:degen_Gaussian}
\end{align}
where $\mathrm{Det}(\cdot)$ is the pseudo-determinant, that is the product of all non-zero eigenvalues.
The density $ f(\x) $ forms an exponential family with natural parameter $Q$, as
\begin{align}
	\x^T\Theta \x &= \sum_{i < j} Q_{ij}(x_i - x_j)^2,\qquad\textrm{Det}(\Theta)=d\sum_{T\in \mathcal{T}} \prod_{ij \in T} Q_{ij},\label{eq:matrix_tree}
\end{align}
where $\mathcal{T}$ is the set of spanning trees over the complete graph with $ d $ vertices and the right hand side in \eqref{eq:matrix_tree} follows from Kirchhoff's matrix-tree theorem.
Here, $ Q $ is the weighted adjacency matrix with zero diagonal as in Section~\ref{s:hr}, i.e.~$ Q_{ij}=-\Theta_{ij} $ for all $ i\neq j $.
Let $ \mathcal{Q}:=\{Q\in\mathbb{S}_0^d: \int_{\mathcal{H}^{d-1}}f(\x)d\x<\infty \} $ denote the natural parameter space and
\begin{align*}
	A(Q) = -\frac{1}{2} \log\left(\sum_{T\in \mathcal{T}} \prod_{ij \in T} Q_{ij}\right).
\end{align*}
the log-partition function of the exponential family.
Let $ \X $ be a random vector distributed according to $ f(\x) $, and assume a sample $ \x_1,\ldots,\x_n \in \mathcal{H}^{d-1} $ of independent observations of $ \X $. 
We obtain from \eqref{eq:matrix_tree} that the sufficient statistic is $\boldsymbol t (\x_1,\ldots,\x_n):= -\frac{1}{2}\widetilde{\Gamma} $ with
\begin{equation}\label{empvario}
	\widetilde{\Gamma}_{ij} := \frac{1}{n} 
	\sum_{\ell=1}^{n} (x_{\ell i} - x_{\ell j})^2,
\end{equation}
where we refer to $ \widetilde{\Gamma} $ as the \textit{sample variogram}.
The mean parameter of the exponential family $ f $ is the mean of the sufficient statistic
$\EE(t(\X))=-\frac{\Gamma}{2},$
where $ \Gamma $ is a one-to-one transformation of $ \Theta $ as described in Section~\ref{s:hr}.
Denote by $ M:=\{\Gamma\in\mathbb{S}_0^d: \int_{\mathcal{H}^{d-1}}f(\x)d\x<\infty \} $ the mean parameter space of the exponential family.
With the Cayley--Menger matrix
$
\CM(\Gamma)= \left(\begin{smallmatrix}
	-\frac{\Gamma}{2}&\bs 1\\
	-\bs 1^T&0\\
\end{smallmatrix}\right)
$
we then obtain that the log-partition function writes in $ \Gamma $ as 
\[A(\Gamma)=\frac{1}{2}\log\det(\CM(\Gamma)), \]
compare for example \citet{REZ2021}. 
Let $ \nabla_{\Gamma},\nabla_{Q} $ be the matrix-valued gradients for variables $ \Gamma,Q\in\mathbb{S}_0^d $.
Standard theory for exponential families gives a direct relation between the log-partition functions and the parameters via
\begin{align}
	\nabla_{\Gamma} A(\Gamma)&=	Q/2,\quad \nabla_{Q}A(Q)=-\Gamma/2, \label{eq:derivatives}
\end{align}
so that the score functions are linear in the parameters, see \citet[Prop.~A.5]{REZ2021}.

For observations of the H\"usler--Reiss vector $ \Y $ we cannot compute $ \widetilde{\Gamma} $ directly, as this would require observations of the extremal function $ \W $.
We therefore approximate $ \widetilde{\Gamma} $ with the \textit{empirical variogram} $ \overline{\Gamma} $ \citep{EV2022}, that we compute from independent observations $ \y_1,\ldots,\y_n $ of $ \Y $ as follows:
Let $ \mathcal{I}_k=\{i\in[n]:y_{ik}>0\} $ denote index sets of observations for which the $ k $th entry exceeds zero. 
When $|\mathcal{I}_k|\ge 2$, define the corresponding subset of centered observations $ (\z_i)_{i\in\mathcal{I}_k} $ via $ \z_i:=\y_{i}-\frac{1}{|\mathcal{I}_k|}\sum_{\ell\in\mathcal{I}_k}\y_{\ell} $.
We obtain a sample variogram  
\[\overline{\Gamma}^{(k)}_{ij}=\frac{1}{|\mathcal{I}_k|}\sum_{\ell\in\mathcal{I}_k} (z_{\ell i} - z_{\ell j})^2. \]
Note that this is the sample variogram of the $ k $th extremal function $ \W^k $.
For $|\mathcal{I}_k|< 2$, we set $ \overline{\Gamma}^{(k)}=\mathbf{0} $.
The empirical variogram is defined as the mean of the sample variograms for all $ k\in[d] $, that is
\[\overline{\Gamma}=\frac{1}{d}\sum_{k=1}^{d}\overline{\Gamma}^{(k)},\]
assuming that $|\mathcal{I}_k|\ge 2$ for at least one $ k\in [d] $.

The empirical variogram allows for the definition of a surrogate log-likelihood for observations of the H\"usler--Reiss vector $ \Y $.
We denote by $\langle  A,B  \rangle := \sum_{i<j} A_{ij} B_{ij}$ the inner product on $\mathbb{S}^d_0$.
The log-likelihood of the exponential family~\eqref{eq:degen_Gaussian} for a given sample variogram $ \widetilde{\Gamma} $ equals $ -\frac{1}{2}\langle \widetilde{\Gamma},Q \rangle-A(Q) $.
Given independent observations $ \y_1,\ldots,\y_n $ of $ \Y $, we approximate the sample variogram $ \widetilde{\Gamma} $ with the empirical variogram $ \overline{\Gamma} $.
This gives rise to the \textit{surrogate log-likelihood} for H\"usler--Reiss distributions that we define as
\begin{align}
	\ell(Q; \overline{\Gamma}) = \frac{1}{2}\log\left(\sum_{T\in \mathcal{T}} \prod_{ij \in T}Q_{ij}\right) - \frac{1}{2}\langle \overline{\Gamma} , Q \rangle. \label{eq:surrogate_likelihood}
\end{align}
In addition to the interpretation as a surrogate for the original log-likelihood of a H\"usler--Reiss distribution, we mention the interpretation of \eqref{eq:surrogate_likelihood} as a degenerate version of the log-determinantal Bregman divergence as in \citet{Ravikumar2011}.

Via \eqref{eq:derivatives} we obtain that the gradient of \eqref{eq:surrogate_likelihood} with respect to $Q$ is $(\Gamma-\overline{\Gamma})/2.$ 
As $ \ell(Q;\overline{\Gamma}) $ is a strictly concave function with respect to $Q$, we find that $ \Gamma=\overline{\Gamma} $ gives the unique maximizer as long as $ \overline{\Gamma}\in M $.
For an arbitrary $ \Gamma^*\in M $, the Fenchel conjugate $A^*(\Gamma^*):= \sup\{\ell (Q; \Gamma^*) : Q \in \mathbb{S}^d_0 \} $ of the log-partition function $A(Q)$ gives the value of the log-likelihood in the unique maximum.
As the unique optimizer of $\ell(Q;\Gamma^*)$ is $Q^*$ we find that
\begin{align*}
	A^*(\Gamma) &=\ell (Q; \Gamma) = - \frac{d-1}{2} -\frac{1}{2}\log\det(\CM(\Gamma)).
\end{align*}
We observe that the Fenchel conjugate $A^*(\Gamma)$ is a smooth, strictly convex function.

The surrogate log-likelihood allows for a simple form of a surrogate Kullback--Leibler divergence.
Consider two H\"usler--Reiss distributions, respectively parameterized by variogram matrices $\Gamma_1$ and $ \Gamma_2 $.
These parameters also characterize two distributions in the degenerate Gaussian exponential family~\eqref{eq:degen_Gaussian}.
Let $ f_1(\x) $ have mean parameter $ \Gamma_1 $ and $ f_2(\x) $ natural parameter $ Q_2 $.
The Kullback--Leibler (KL) divergence between these two distributions \citep[Proposition~6.3]{Brown1986} is given by
\begin{align}\label{kldivergence}
	\textrm{KL}(\Gamma_1, Q_2) &= \frac{1}{2} \langle \Gamma_1, Q_2 \rangle + A^*(\Gamma_1) + A(Q_2).
\end{align}
We consider $ \textrm{KL}(\Gamma_1, Q_2) $ as a surrogate KL-divergence for H\"usler--Reiss distributions.
If we ignore constants and summands that do not depend on $ \Gamma_1 $ and $ Q_2 $, we find that
\begin{align}
	\textrm{KL}(\Gamma_1, Q_2)
	&\;\propto\;  \langle \Gamma_1, Q_2 \rangle -\log\det(\CM(\Gamma_1)) - \log\left(\sum_{T\in \mathcal{T}} \prod_{ij \in T}(Q_2)_{ij}\right). \label{eq:KL_proportional}
\end{align}
Given that $ \langle \Gamma_1, Q_2 \rangle $ is linear in both $\Gamma_1$ and $Q_2$, that $A(Q_2)$ is strictly convex in $Q_2$ and that $A^*(\Gamma_1)$ is strictly convex in $\Gamma_1$, the surrogate KL-divergence is strictly convex in both $\Gamma_1$ and $Q_2$.
Note that $\textrm{KL}(\Gamma_1, Q_2)$ is well defined on $M \times \mathcal{Q}$ and $\textrm{KL}(\Gamma_1, Q_2) = 0$ if and only if the two distributions are the same, i.e.~$\Gamma_1 = \Gamma_2$.

\subsection{Surrogate mixed dual estimation for H\"usler--Reiss distributions}

Our interest is in parameter estimation for locally metrical H\"usler--Reiss graphical models.
This invokes constraints on both the variogram matrix $ \Gamma $ and the weighted adjacency matrix $ Q $.
In terms of the surrogate log-likelihood, this means that we impose constraints on both the natural parameter space $ \mathcal{Q} $ and the mean parameter space $ M $.
Rewriting these constraints in terms of either the natural parameter or the mean parameter leads to highly nonlinear parametric constraints and a potentially non-convex optimization problem.
We tackle this problem following the new approach of \citet{LZ2022}, which proposes to solve two consecutive convex optimization problems in terms of a mixed convex exponential family. 
This results in what they call a mixed dual estimator (MDE) that incorporates constraints on both the mean parameter and the natural parameter for exponential families.
Consider the sufficient statistics $\boldsymbol t$ as being the concatenation of two sub-vectors, so that $\boldsymbol t(\boldsymbol x) = (\boldsymbol u, \boldsymbol v)$, where $\boldsymbol u$ and $\boldsymbol v$ are of dimension $r$ and $s$ with $r+s = \binom{d}{2}$. 
Let the mean and natural parameters of the exponential family~\eqref{eq:degen_Gaussian} split accordingly, which we denote as follows:
$$Q = (Q_u, Q_v), \quad \Gamma = (\Gamma_u, \Gamma_v).$$
Hereby $(\Gamma_u, Q_v) \in M_u \times \mathcal{Q}_v$, where $M_u$ is the projection of $M$ on $\Gamma_u$ and $\mathcal{Q}_v$ the projection of $\mathcal{Q}$ on $Q_v$. 
According to \citet{BN1978}, the mixed parameterization $(\Gamma_u, Q_v)$ is a valid parameterization of the exponential family. 
We refer to \citet{BN1978} and \cite{LZ2022} for details on mixed exponential families.

The MDE introduced in \citet{LZ2022} relies on solving two consecutive convex problems. 
It minimizes the Kullback--Leibler divergence \citep{KL1951} between members of the same exponential family in two steps.
Let $ C_v\subseteq\mathcal{Q} $ denote convex constraints on $ \Gamma_{v} $ and $ C_u\subseteq M $ convex constraints on $ Q_u $.
\begin{defn}\label{def:MDE}
	Let $ \widetilde{\Gamma} $ be a sample variogram. Define a two-step estimation procedure such that
	\begin{enumerate}
		\item the first step~(S1) is defined as
		\[\widehat{Q}=\argmin_{Q \in C_v} \textrm{KL}(\widetilde{\Gamma}, Q), \] 
		\item and the second step~(S2) as
		\[\widecheck{\Gamma}=\argmin_{\Gamma\in C_u}\textrm{KL}(\Gamma, \widehat{Q}).\]
	\end{enumerate}
	Then we call $ \widecheck{\Gamma} $ a \textit{mixed dual estimator} for the convex constraints given by $ C_u $ and $ C_v $.
\end{defn}
Both steps of the mixed dual estimator are convex problems since $\textrm{KL}(\widetilde{\Gamma}, Q)$ is convex in $Q \in C_v$ and $\textrm{KL}(\Gamma, \widehat{Q})$ is convex in $ \Gamma \in C_u $.
Note that the unique optimum of step~(S1) is the maximum likelihood estimator (MLE) for the convex exponential family given by $Q \in C_v$.
\begin{remark}
	By \citet[Proposition 4.3]{LZ2022}, if the unique optimum $\widehat{Q}$ in the first step exists, then the unique optimum $\widecheck{\Gamma}$ in the second step necessarily exists. 
	In this case, \citet[Theorem 4.5]{LZ2022} gives that $\widehat{\Gamma}_u = \widetilde{\Gamma}_u \in M_u$, and  $\widecheck{Q}_v =\widehat{Q}_v \in \mathcal{Q}_v$. Furthermore, $\widecheck{\Gamma}$ is a valid parameter of the mixed exponential family.
\end{remark} 
\begin{remark}
	\citet{LZ2022} also studied the asymptotic properties of the MDE and compared them to those of the MLE. 
	Assume the MDE for the exponential family~\eqref{eq:degen_Gaussian} obtained from the mixed dual estimation procedure described in Definition~\ref{def:MDE} above with a sample of size~$n$. 
	Plus, consider the MLE of the mixed exponential family with parametrization $ (\Gamma_u,Q_v) $ under the constraints ${\Gamma_u} \in C_u$ and $Q_v \in C_v$ given data $ (\widetilde{\Gamma}_u,\widetilde{Q}_v) $. 
	It follows that the MDE $(\widecheck{\Gamma}_u,\widecheck{Q}_v)$ is $\sqrt{n}$-consistent and has the same asymptotic distribution as the MLE.
\end{remark}

\subsection{Mixed dual estimator for locally metrical H\"usler--Reiss graphical models}\label{s:2_step}
The mixed dual estimator is of particular interest for us because the surrogate likelihood of locally metrical H\"usler--Reiss graphical models allows a representation as a mixed convex exponential family.
Let $ G=(V,\mathcal{E}) $ be an undirected graph.
The constraints for the local metric property with respect to $ G $ (see Definition~\ref{def:extremlocmetrHRGM}) are given by
\begin{align*}
	C_u&=\{\Gamma\in M:\Gamma_{ij} \leq \Gamma_{ik} + \Gamma_{jk},\; \forall i,j,k \in C,\; \forall C \in \mathcal{C}(G)\},\\
	C_v&=\{Q\in\mathcal{Q}:Q_{ij}=0,\; \forall ij\not\in \mathcal{E} \}.
\end{align*}
These constraints are convex and allow for an MDE that we will now study in detail.
\subsubsection{Step~(S1)}
Let $ \y_1,\ldots\y_n $ of $ n $ be a sample of independent observations of a H\"usler--Reiss vector $ \Y $ with parameter matrix $ \Gamma $.
We approximate the sample variogram $ \widetilde{\Gamma} $ by the empirical variogram $ \overline{\Gamma} $ (see Section~\ref{sec:surrogate}), so that the optimization problem for step~(S1) is to minimize of the KL-divergence between $\overline{\Gamma}$ and the parameter space $C_v$. This simplifies to
\begin{equation}
	\begin{aligned}
		& \underset{Q \in C_v}{\text{minimize}}
		& & \langle \overline{\Gamma}, Q \rangle - \log\left(\sum_{T\in \mathcal{T}} \prod_{ij \in T}Q_{ij}\right),\\
	\end{aligned}
	\label{eq:sn4}
\end{equation}
since removing the terms constant with respect to $ Q $ does not affect the location of the optimum. 
The estimation in step~(S1) yields the surrogate MLE of a H\"usler--Reiss graphical model on a graph $G = (V,\mathcal{E})$ as in \citet[Proposition~1]{HES2022}.
Standard exponential family theory shows that the system
\begin{align}
	\begin{cases}
		\widehat{\Gamma}_{ij}=\overline{\Gamma}_{ij},\; &ij\in \mathcal{E},\\
		\widehat{Q}_{ij}=0,\; &ij \not\in \mathcal{E},
	\end{cases}\label{eq:matrix_compl}
\end{align}
constrained such that $ \widehat{\Gamma} $ is conditionally negative definite, gives the unique solution of \eqref{eq:sn4}.
We can therefore use the matrix completion procedure for variogram matrices presented by \citet{HES2022} to obtain the step~(S1) estimate.

\subsubsection{Step~(S2)}
Assume a step~(S1) estimate $\widehat{Q}$ and parametric constraints $ C_u $ as defined above.
Removing terms constant with respect to $ \Gamma $ in the KL-divergence, the minimization problem for step~(S2) is
\begin{equation}
	\begin{aligned}
		& \underset{\Gamma \in \mathcal{C}^d}{\text{minimize}}
		& & \langle  \Gamma, \widehat{Q} \rangle - \log\det(\CM(\Gamma))\\
		& \text{subject to}
		& & \Gamma_{ij} \leq \Gamma_{ik} + \Gamma_{jk} \quad \forall i,j,k \in C,\; \forall C \in \mathcal{C}(G),
	\end{aligned}
	\label{optprob}
\end{equation}
where the inequality constraints on $\Gamma$ guarantee the local metric property for H\"usler--Reiss distributions as given in \eqref{locmetrGamma}.
Note that the objective function in \eqref{optprob} is the negative of the dual or reciprocal surrogate likelihood
\begin{align}
	\widecheck{\ell}(\Gamma; \widehat{Q})&:=A^*(\Gamma)-\langle {\Gamma}, \widehat{Q} \rangle\;\propto\; \log\det(\CM(\Gamma))-\langle {\Gamma}, \widehat{Q} \rangle.\label{eq:dual_loglik}
\end{align}
Therefore, the second step~(S2) can be seen as a surrogate maximum reciprocal likelihood problem.
We will now study the optimization problem~\eqref{optprob} in detail.

\subsection{Dual of step~(S2)}
We derive the dual and the Karush--Kuhn--Tucker (KKT) conditions for the optimization problem \eqref{optprob}, where we will employ the latter to certify optimality of a solution. This relies on the theory of convex optimization, we refer to \citet{Boyd2004} for details. 

Assume without loss of generality that the graph $G=(V,\mathcal{E})$ in the optimization problem \eqref{optprob} has at least one clique of size greater or equal to three. 
This means that there is at least one triangle in the graph, where with triangles we mean all triples of nodes that form fully connected subgraphs.
Let $ \Delta $ denote the set of all distinct triangles in $ G $, so that $ i<j<k $ for each $ (i,j,k)\in\Delta $.
For each triangle $ (i,j,k)\in\Delta $, there are three triangle inequalities that we abbreviate as
\begin{align*}
	g_{ijk} &:= \Gamma_{ij} -\Gamma_{ik} -\Gamma_{jk} \leq 0,\\
	g_{ikj} &:= \Gamma_{ik} -\Gamma_{ij} -\Gamma_{jk}\leq 0,\\
	g_{jki} &:= \Gamma_{jk} -\Gamma_{ij} -\Gamma_{ik}\leq 0.
\end{align*}
Therefore the total number of constraints that encode the local metric property for a graph $G$ equals $ m=3|\Delta| $, and  $ g $ is a vector in $ \RR^m $. We enumerate the entries of $ g $ in lexicographic order according to the index set $ \Psi=\bigcup_{(i,j,k)\in\Delta}\{(i,j,k),(i,k,j),(j,k,i)\} $.

The KKT multipliers are constants appearing in the Lagrangian function of the optimization problem, where each multiplier corresponds to an inequality constraint.
Let therefore $ \boldsymbol{\eta}\in\RR^m $ denote the vector of KKT multipliers.
We enumerate the entries of $ \boldsymbol{\eta} $ according to $ \Psi $.
The following proposition gives the dual problem of \eqref{optprob} in terms of the KKT multipliers $ \boldsymbol{\eta} $:
\begin{prop}\label{prop:dual}
	The dual problem of the step~(S2) optimization problem \eqref{optprob} is
	\begin{align}
		\widecheck{\boldsymbol{\eta}}=\argmax_{\boldsymbol{\eta}\ge \mathbf{0}}\; \log\left(\sum_{T\in \mathcal{T}} \prod_{ij \in T}Q_{ij}(\boldsymbol{\eta})\right)+(d-1).\label{eq:dualproblem}
	\end{align}
\end{prop}	
A proof of Proposition~\ref{prop:dual} is available in Appendix~\ref{app:pr_dual}.
The KKT conditions allow to certify optimality:
\begin{thm}
	A point $ (\Gamma(\widecheck{\boldsymbol{\eta}}),\widecheck{\boldsymbol{\eta}}) $ is the unique solution of \eqref{infsup} if and only if
	\begin{enumerate}
		\item[(i)] $ \Gamma_{ij}(\widecheck{\boldsymbol{\eta}})\le \Gamma_{ik}(\widecheck{\boldsymbol{\eta}})+\Gamma_{jk}(\widecheck{\boldsymbol{\eta}}) $ for all $ (i,j,k)\in \Psi $,
		\item[(ii)] $ \widecheck{\boldsymbol{\eta}}\ge \mathbf{0} $,
		\item[(iii)] $ \sum_{(i,j,k)\in \Psi}\widecheck{\eta}_{ijk}(\Gamma_{ij}(\widecheck{\boldsymbol{\eta}})- \Gamma_{ik}(\widecheck{\boldsymbol{\eta}})-\Gamma_{jk}(\widecheck{\boldsymbol{\eta}}))=0. $
	\end{enumerate}
\end{thm}
Note that the unique solution of \eqref{infsup} is equivalent to the unique solution of the primal problem \eqref{optprob}, hence certifying optimality for the MDE.
\begin{remark}
	As mentioned previously, since the optimization problem in step~(S2) is a convex problem with only affine inequality constraints, strong duality holds. 
	This means that the duality gap, which is the difference between the two objective functions in (\ref{optprob}) and \eqref{eq:dualproblem} for the same evaluation point is 
	\begin{equation}
		\langle  \Gamma(\boldsymbol{\eta}), \widehat{Q} \rangle - \log\det(\CM(\Gamma(\boldsymbol{\eta}))) - \log\left(\sum_{T\in \mathcal{T}} \prod_{ij \in T}Q_{ij}(\boldsymbol{\eta})\right)- (d-1)=	 \langle \Gamma(\boldsymbol{\eta}), \widehat{Q}  \rangle - (d-1).\label{eq:dual_gap}
	\end{equation}
	By definition, the duality gap for convex problems must be zero at the optimum $\widecheck{\boldsymbol{\eta}}$. We will therefore be able to track the duality gap and thus to have a necessary condition for convergence of our optimization algorithm.
\end{remark}

\section{An algorithm for the two-step estimator}\label{sec:algorithm}

Our objective in this section is to develop an algorithm that solves the two-step optimization problem described in Section~\ref{s:2_step} given a graph $ G=(V,\mathcal{E}) $ and an empirical variogram $ \overline{\Gamma} $.
As step~(S1) corresponds to the surrogate maximum likelihood estimator for the H\"usler--Reiss graphical model with respect to $ G $, one can employ for example the matrix completion algorithm of \citet{HES2022} to find $ \widehat{Q} $.
For the optimization problem in step~(S2), we employ a gradient descent algorithm for the dual problem~\eqref{eq:dualproblem}. 
It returns the MDE $\widecheck{\boldsymbol{\eta}}$, whose related parameter $\widecheck{\Gamma}$ parameterizes a H\"usler--Reiss graphical model that satisfies the local metric property with respect to the graph~$G$.

\subsection{An algorithm for the optimization problem in step~(S1)}\label{1step1}
The optimizer of the problem \eqref{eq:sn4} for step~(S1) is the surrogate maximum likelihood estimator for a H\"usler--Reiss graphical model with respect to the graph $ G $. 
For this problem, \citet{HES2022} developed a matrix completion algorithm that solves \eqref{eq:matrix_compl} for a given sample variogram and graph. 
An implementation of the algorithm is available in the R package \texttt{graphicalExtremes} in the function \texttt{complete\_Gamma}. 
The algorithm returns an estimate $\widehat{\Gamma}$ from which we compute the starting point $ \widehat{Q} $ for step~(S2).

\subsection{An algorithm for the optimization problem in step~(S2)}
The second step gives the MDE for locally metrical H\"usler--Reiss graphical models. 
We propose an optimization algorithm on the dual problem \eqref{eq:dualproblem} which sequentially updates the parameter vector $\boldsymbol{\eta}$ using the method of moving asymptotes (MMA) \citep{Svanberg2001} in the R package \texttt{nloptr} \citep{nlopt}. 
The MMA algorithm requires to evaluate the following functions for a given $\boldsymbol{\eta}$: 
\begin{enumerate}
	\item the objective function of \eqref{eq:dualproblem} to minimize, 
	\item its gradient with respect to $\boldsymbol{\eta}$.
\end{enumerate}

In order to improve the efficiency of our algorithm by avoiding loops when evaluating the summation in equation \eqref{sollagrangian}, we vectorize it. 
First, consider the vector $\overrightarrow{Q} \in \mathbb{R}^{d(d-1)/2}$ which is the vectorized analogue of $Q \in \mathbb{S}^{d}_0$ obtained by stacking the upper triangular elements of $Q$ in lexicographic order. 
The vector $\overrightarrow{\widehat{Q}}$ is defined analogously. We can consider the following example:
\begin{equation*}
	Q =\begin{bmatrix}
		0& Q_{12}& Q_{13}& Q_{14}\\
		Q_{12}& 0& Q_{23}& Q_{24}\\
		Q_{13}& Q_{23}& 0& Q_{34}\\
		Q_{14}& Q_{24}& Q_{34}& 0
	\end{bmatrix} \in \mathbb{S}^4_0 \quad
	\longrightarrow \quad \overrightarrow{Q} = \begin{bmatrix}
		Q_{12}\\
		Q_{13}\\
		Q_{14}\\
		Q_{23}\\
		Q_{24}\\
		Q_{34}
	\end{bmatrix}.
\end{equation*}
The vectorized version of the second summand of the right hand side in \eqref{sollagrangian} is 
\begin{equation}\label{vectorize}
	\begin{bmatrix} 
		\sum_{(i,j,k) \in \Psi} \eta_{ijk} \frac{\partial g_{ijk}}{\partial \Gamma_{12}}\\  
		...\\  
		\sum_{(i,j,k) \in \Psi} \eta_{ijk} \frac{\partial g_{ijk}}{\partial \Gamma_{(d-1)d}} 
	\end{bmatrix} \in \mathbb{R}^{d(d-1)/2}.
\end{equation}
This requires to evaluate the triangle inequality constraint for all triples $ (i,j,k) \in \Psi $. 
Each element of the vector in \eqref{vectorize} is a scalar product of ${\boldsymbol \eta}$ and a partial derivative of the vector $ g $, that is a vector with entries equal to 1, 0, or -1. 
Therefore, we can rewrite the vector in \eqref{vectorize} as the matrix product of a matrix $ \A(G) \in \mathbb{R}^{d(d-1)/2 \times m}$ and the vector $\boldsymbol{\eta} \in \mathbb{R}^m$.
Let $ g $ be the vector of constraints $ g_{ijk} $ in lexicographic order.
The rows of $\A(G)$ are indexed by all pairs $ i<j $, so that the $ ij $-th row equals $ A_{ij}(G)=(\frac{\partial}{\partial \Gamma_{ij}}g)^T $.
It follows that
\begin{equation}
	\begin{bmatrix} 
		\sum_{(i,j,k) \in \Psi} \eta_{ijk} \frac{\partial g_{ijk}}{\partial \Gamma_{12}}\\  
		...\\  
		\sum_{(i,j,k) \in \Psi} \eta_{ijk} \frac{\partial g_{ijk}}{\partial \Gamma_{(d-1)d}} 
	\end{bmatrix} = \A(G) \boldsymbol{\eta}.
\end{equation}
Note that the $\A(G)$ matrix is fully determined for a given graph $G$. For ease of notation, we will refer to it as $\A$ instead of $\A(G)$ since the graph in question is clear.
\begin{example}
	Consider the graphical structure in Figure \ref{fig:graph_HR} and consider a H\"usler--Reiss graphical model $\boldsymbol Y$ on the graph $G$ parameterized by the variogram matrix $\Gamma$. Then, the related $A$ matrix and the vector ${\boldsymbol \eta}$ are given by
	$$\A =
	\begin{pmatrix}
		&1& -1& -1&0& 0& 0&  \\
		&-1& 1& -1&0& 0& 0&  \\
		&0&0& 0& 0& 0& 0&  \\
		&-1& -1& 1&  1&-1& -1 \\
		&0& 0& 0& -1& 1& -1 \\
		&0& 0& 0& -1&-1& 1
	\end{pmatrix}, \quad
	{\boldsymbol \eta} =
	\begin{bmatrix}
		\eta_{123}\\ \eta_{132}\\ \eta_{231}\\
		\eta_{234}\\ \eta_{243}\\ \eta_{342}
	\end{bmatrix},
	$$
	where the rows of $A$ correspond to all elements $\Gamma_{12},...,\Gamma_{34}$ of the $\Gamma$ matrix, and each column to a certain triangle inequality constraint $g_{ijk}$. 
	For example, the first column represents the constraint $g_{123}:= \Gamma_{12} - \Gamma_{13} - \Gamma_{23} \leq 0$, and the three first columns represents the inequality constraints necessary to have an associated marginal graphical model $\Y_{123}$ on the triangular graph $G_{123}$.
\end{example}
We can now express equation (\ref{sollagrangian}) in its vectorized form
\begin{equation*}
	\overrightarrow{Q}(\boldsymbol{\eta}) = \overrightarrow{\widehat{Q}} +\A\boldsymbol{\eta}, 
\end{equation*}
and we can evaluate the objective function of the dual problem \eqref{eq:dualproblem} for a given vector $\boldsymbol{\eta}$.
Therefore, we can provide for a given $\boldsymbol{\eta}$: the objective function of the dual problem \eqref{eq:dualproblem} and its gradient with respect to $ \boldsymbol{\eta}$ that is
\begin{equation}\label{eq:gradient}
	\nabla_{\boldsymbol{\eta}}\left[\log\left(\sum_{T\in \mathcal{T}} \prod_{ij \in T}Q_{ij}(\boldsymbol{\eta})\right)+(d-1)\right]=\A^T \overrightarrow{\Gamma(\boldsymbol{\eta})},
\end{equation}
and the inequality constraint which is simply $\boldsymbol{\eta} \geq {\boldsymbol 0}$.
\begin{algorithm} This algorithm solves the dual problem as given in Proposition~\ref{prop:dual}.
	\label{alg:mpHR}
	$ $ \\[1ex]
	\textbf{Input:} {A variogram $ \widehat{\Gamma} $ and a graph $ G $.}
	
	\noindent
	\textbf{Initialize:} Compute the initial value $ \widehat{Q} $ and the constraint matrix $ \A $.
	Define the objective function as in \eqref{eq:dualproblem}, the nonnegativity constraints $ \boldsymbol{\eta}\ge \mathbf{0} $ and the gradient of the objective function \eqref{eq:gradient}.
	
	\noindent
	\textbf{Computation:} Solve the optimization problem in $ \boldsymbol{\eta} $ using the \texttt{MMA} algorithm.
	
	\noindent
	\textbf{Output: }{The optimal variogram under the local metric property.}
\end{algorithm}

\subsection{Performance}

We study the performance of Algorithm~\ref{alg:mpHR} on simulated data.
For this, we sample $ n_{\text{sim}}=100 $ times a variogram matrix $ \overline{\Gamma} $ and a sparse connected graph for $ d\in\{20,50,100\} $ vertices, so that we obtain a step~(S1) estimate via the function \texttt{complete\_Gamma} from the \texttt{R} package \texttt{graphicalExtremes}.
For the variogram matrix, we use that any variogram matrix is a Euclidean distance matrix.
Therefore we sample uniformly $ d $ points on the $ (d-1) $-dimensional sphere, and compute the Euclidean distance matrix from these points.
Afterwards, we generate a sparse connected graph that contains at least one triangle using the function \textit{generate\_random\_connected\_graph} from the \texttt{R} package \texttt{graphicalExtremes}, and compute the step~(S1) estimate $ (\widehat{\Gamma},\widehat{Q}) $ as explained above.
In Table~\ref{tab:sim_study} we show the rounded average values from $ n_{\text{sim}}=100 $ simulations for the number of inequalities for the local metric property, that is the dimension of the dual problem, and the proportion of inequalities that hold already for the step~(S1) estimate.
Furthermore, we give the average of the duality gaps \eqref{eq:dual_gap}.
We observe that the algorithm is very fast for a dual dimension of $ m=468 $, and converges with good precision in less than 15 seconds even for an average dual dimension of $ m=3844 $.
The simulations were conducted on a standard laptop.
\begin{table}[ht]
	\centering
	\begin{tabular}{rrrr}
		\hline
		$ d $& $ 20 $ & $ 50 $ & $ 100 $ \\ 
		\hline
		Dual dimension & 29.79 & 468.30 & 3844.26 \\ 
		Start proportion (in \%) & 90.06 & 94.85 & 96.61 \\ 
		Duality gap & -2.20e-09 &-2.79e-08 &-1.08e-06\\
		Time (in s) & 0.05 & 1.04 & 14.95 \\ 
		\hline
	\end{tabular}
	\caption{Averaged results for $ n_{\text{sim}}=100 $ simulated graphical models.}\label{tab:sim_study}
\end{table}

\section{Application: Flight delays data}\label{sec:application}

We study the flight delays data set of \citet{HES2022}. The data set contains (positive) total daily delays in minutes for airports in the United States between 2005 and 2020.
As extreme delays have a variety of negative impacts, our interest is in modeling and estimating the extremal conditional independence structure.
Due to the network structure of air traffic systems we expect local positive dependence for highly connected parts of the network.

\subsection{Data}
The pre-processed flight delays data set is available in the \texttt{R} package \texttt{graphicalExtremes} \citep{graphicalExtremespackage}.
In our study we include all airports with at least 2000 outgoing and incoming flights per year, so that we consider data for $ d=79 $ airports.
These airports are displayed on the map in Figure~\ref{fig:airports}.
\begin{figure}
	\includegraphics[scale=0.7]{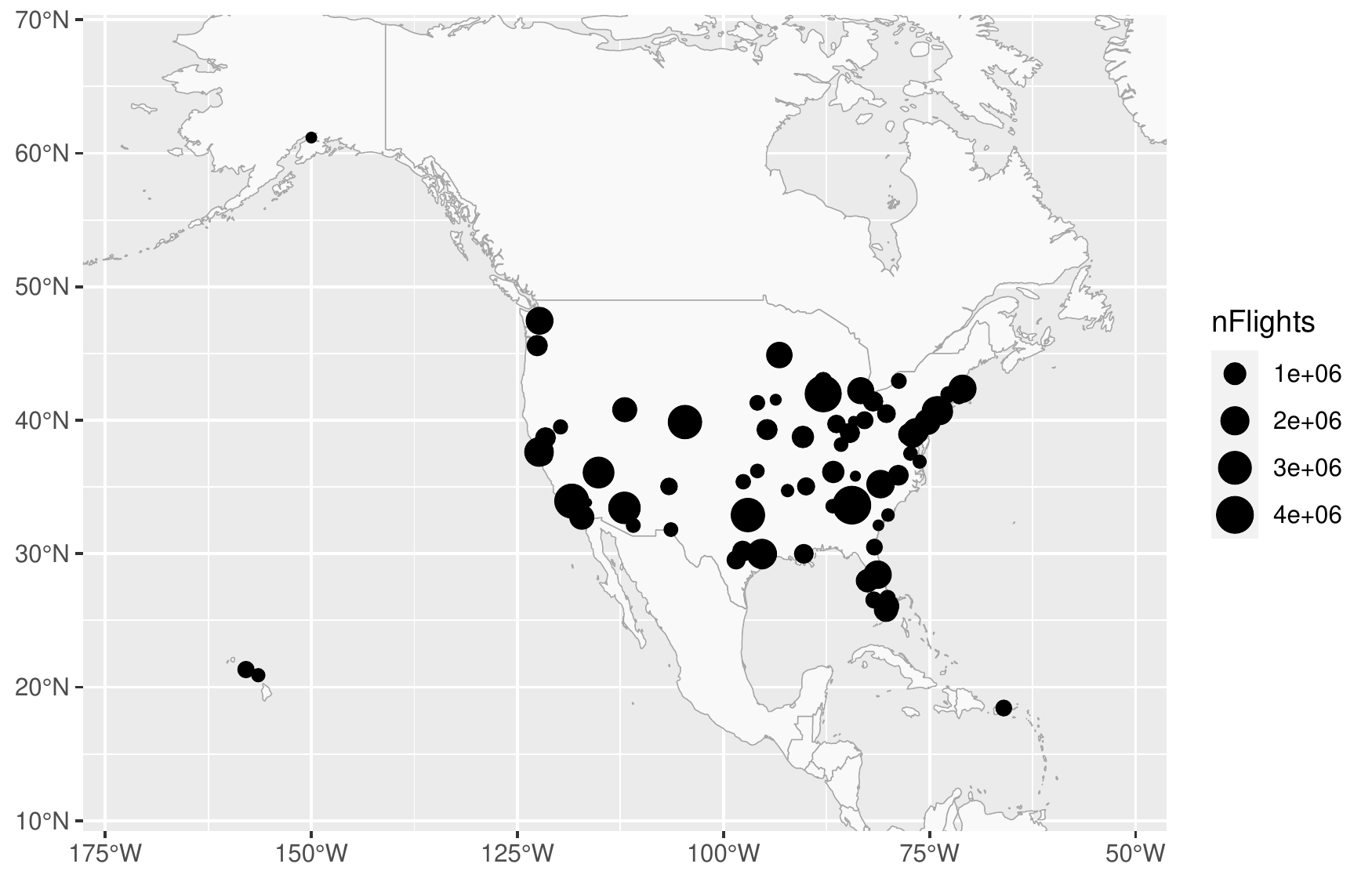}
	\caption{Airports with more than 2000 incoming and outgoing flights per year, with total number of flights indicated by dot size.}\label{fig:airports}
\end{figure}
We remove any day with missing data, which amounts to $ n=5347 $ observations.
The data is split into a training data set for delays between 2005 and 2010 and a validation data set for delays between 2011 and 2020.
We compute the empirical variogram $ \overline{\Gamma} $ for a probability threshold of $ p=0.85 $ for the training data.

\subsection{Exploratory analysis}
In exploratory analysis under the assumption of a H\"usler--Reiss distribution, we find that the empirical variogram satisfies the triangle inequalities for all but 90 of the $ 237237 $ inequalities ($ \approx99,96\% $), which gives good evidence for intrinsic local positivity in the data.
In comparison, only 53,78\% of the non-diagonal entries of the empirical H\"usler--Reiss precision matrix $ \overline{\Theta} $ are non-positive, so that the assumption of \EMTPtwo could be too strong for this data set.

\subsection{Learning locally metrical H\"usler--Reiss graphical models}
To apply our two-step algorithm to learn locally metrical H\"usler--Reiss graphical models for the flight delay data set, we first require a graph estimate for the graphical model.
Therefore we use the \texttt{eglearn} algorithm of \citep{ELV2021} for graph structure learning, i.e.~to find a list of graph estimates given a list of regularization parameters.
For each regularization parameter in \texttt{eglearn}, we compute the first step estimate $ (\widehat{\Gamma},\widehat{Q}) $ for the corresponding graph estimate using the \texttt{complete\_Gamma} algorithm from the \texttt{R} package \texttt{graphicalExtremes} \citep{graphicalExtremespackage}.
We apply Algorithm~\ref{alg:mpHR} to obtain the second step estimate $ (\widecheck{\Gamma},\widecheck{Q}) $ to enforce the local metric property with respect to the corresponding graph.

Both the step~(S1) and step~(S2) estimates are evaluated via their H\"usler--Reiss log-likelihood on the validation data for all regularization parameters $ \rho=0,0.01,\ldots,0.19,0.2 $, see Figure~\ref{fig:plot_lik}.
For $ 0\le\rho\le 0.06 $, we observe that enforcing the local metric property only marginally decreases the log-likelihood on the validation data.
Furthermore, we find that for all $ \rho\ge 0.07 $ the step~(S1) estimates already satisfy the local metric property with respect to the corresponding graph estimates. 
\begin{figure}
	\includegraphics[scale=.7]{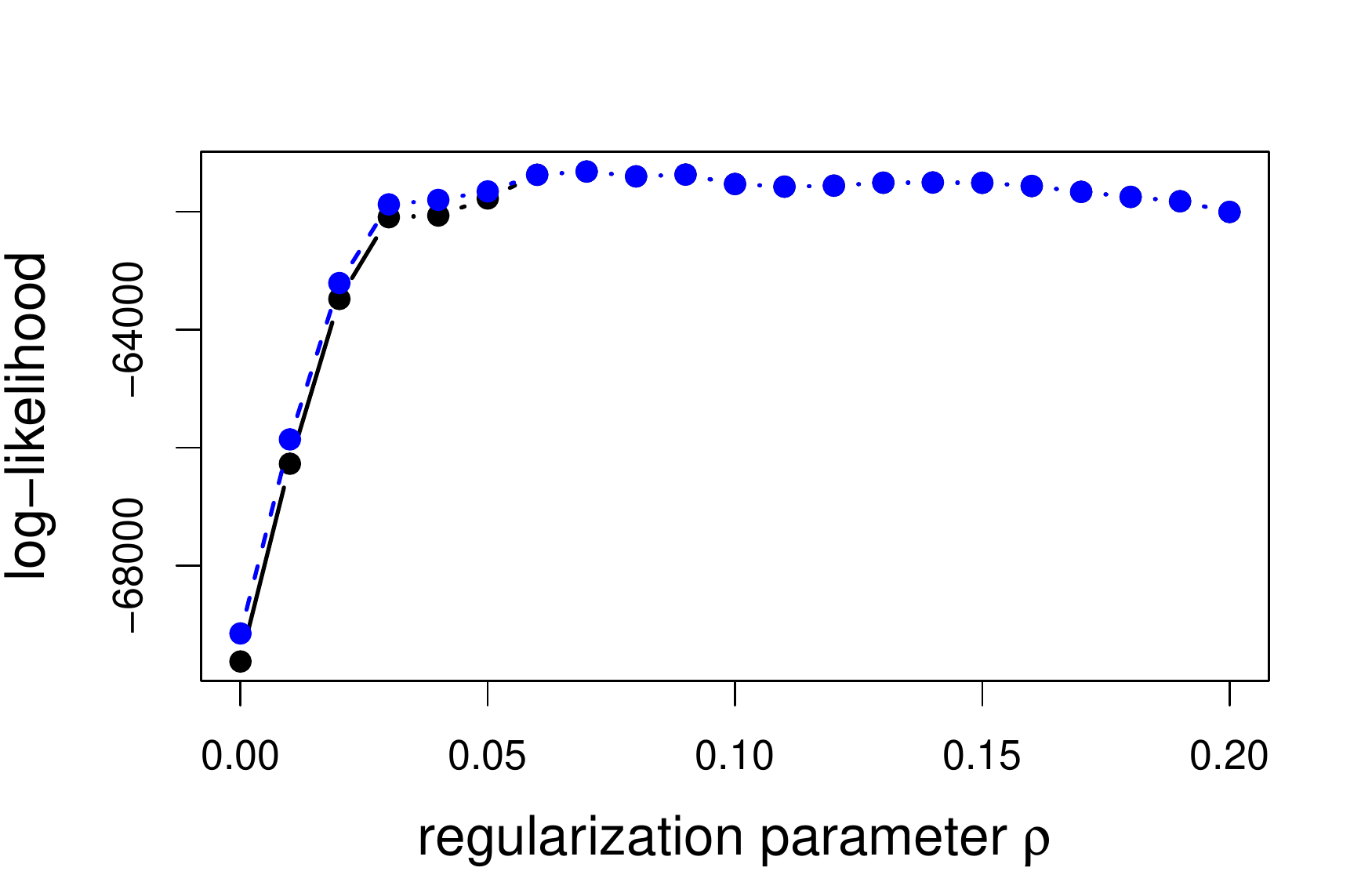}
	\caption{H\"usler--Reiss log-likelihood on the validation data for the step~(S1) estimates (blue dots, connected by dashed lines) and step~(S2) estimates (black, solid lines) for regularization parameters $ \rho = 0,0.01,\ldots,0.19,0.2 $. The step~(S1) estimate for $ \rho=0 $ is the empirical variogram for the training data. The step~(S2) estimate is identical to the step~(S1) estimates for all $ \rho\ge 0.07 $.}\label{fig:plot_lik}
\end{figure}
This includes the estimate with the highest log-likelihood on the validation data at $ \rho=0.08 $.
A map with the graph estimate for $ \rho=0.08 $ is shown in Figure~\ref{fig:best_graph}.
\begin{figure}[h]
	\includegraphics[scale=0.7]{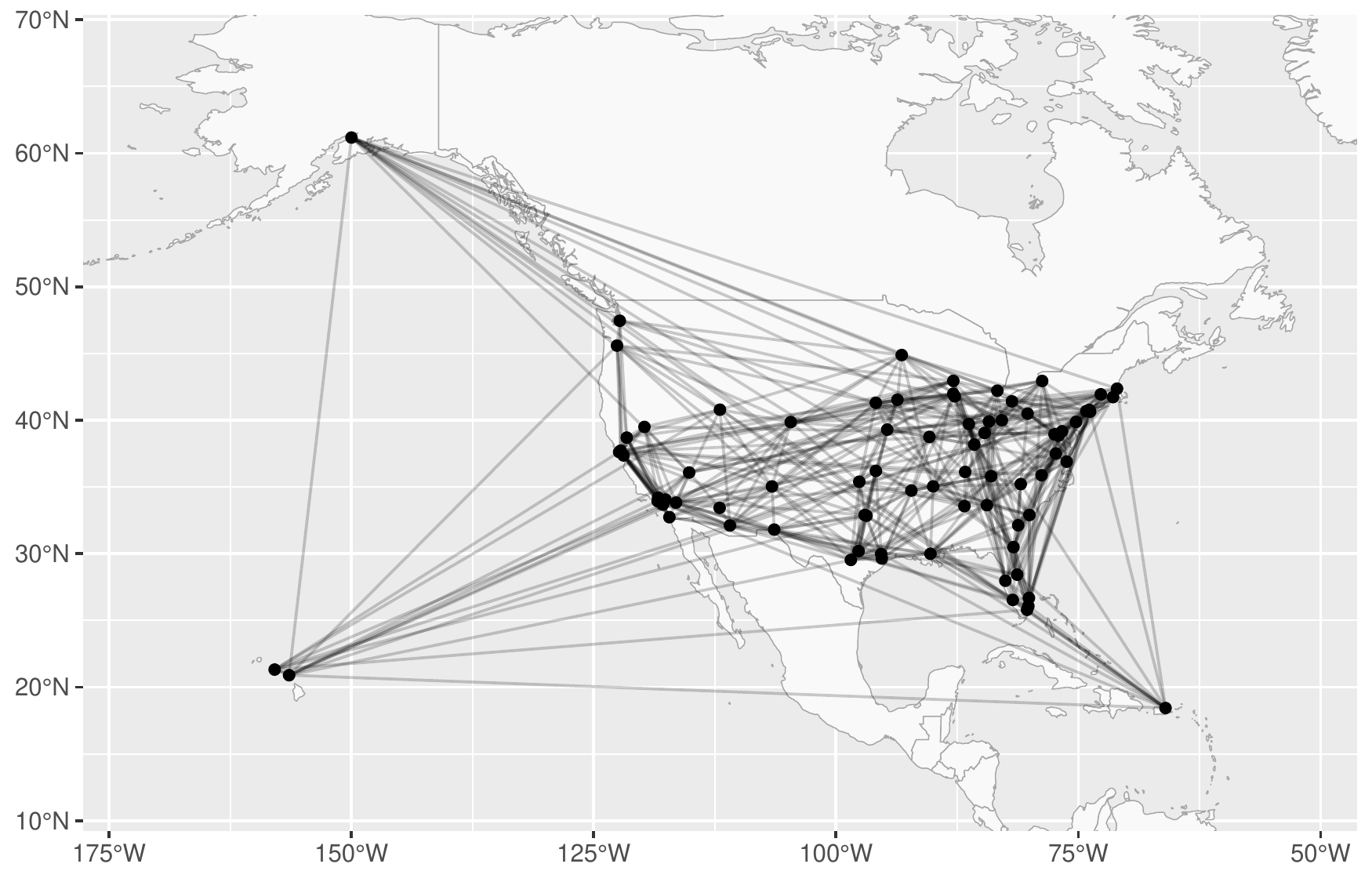}
	\caption{The map shows the graph estimate with the largest log-likelihood. The graph contains 486 edges. The corresponding step~(S1) estimate already satisfies the local metric property and is therefore identical to the step~(S2) estimate.}\label{fig:best_graph}
\end{figure}
This estimate contains 486 out of 3081 possible edges. The step~(S1) estimate satisfies the local metric property with respect to the graph, so that the step~(S2) estimate is identical to the step~(S1) estimate. 
Note that this renders the second step redundant.

In summary, we observed that for larger penalization parameters, i.e.~for sparser graph estimates, the first step estimates already satisfied the local metric property, confirming our expectation of local positive dependence for highly connected subgroups of airports. For coarser graph estimates, the second step only marginally decreased the log-likelihood, while allowing for a local metric interpretation of the parameter estimate.
We conclude that the local metric property is a very reasonable assumption for modeling the flights data set.

\section*{Acknowledgments}
The authors would like to thank Sebastian Engelke and Manuel Hentschel for very helpful discussions and support.
Frank R\"ottger was supported by the Swiss National Science Foundation (Grant 186858).

\appendix
\section{Proofs}\label{app:proofs}

\subsection{Proof of Proposition~\ref{prop:dual}}\label{app:pr_dual}
We define a function which incorporate the unconstrained dual surrogate log-likelihood function \eqref{eq:dual_loglik} of $\Gamma$ given $\widehat{Q}$ and the constraint for $\Gamma$ to be conditionally negative definite, i.e.
\begin{equation*}
	f(\Gamma):= \left\{ \begin{array}{rcl}
		-\widecheck{\ell}(\Gamma; \widehat{Q}), & \mbox{for}\;  \Gamma \in \mathcal{C}^d
		\\ +\infty, & \mbox{otherwise.} 
	\end{array}\right.
\end{equation*}
Thus the problem \eqref{optprob} is equivalent to
\begin{equation*}
	\begin{aligned}
		& \underset{\Gamma \in \mathcal{C}^d}{\text{minimize}}
		& & f(\Gamma)\\
		& \text{subject to}
		& & g_{ijk} \leq 0 \quad \forall (i,j,k) \in \Psi,
	\end{aligned}
\end{equation*}
and the Lagrangian for this problem is
\begin{align}
	L(\Gamma, \boldsymbol \eta) & = f(\Gamma) + \sum_{(i,j,k) \in \Psi}\left( \eta_{ijk} g_{ijk} + \eta_{kij} g_{kij} +\eta_{jki} g_{jki} \right).  \label{lagrangeprimal}
\end{align}
Given the Lagrangian above, we define the Fenchel conjugate of $f(\Gamma)$, which we denote $f^*(\boldsymbol \eta)$:
\begin{align*} 
	f^*(\boldsymbol \eta) & = \sup_{\boldsymbol\eta \geq \mathbf{0}} L(\Gamma,\eta)  =  \begin{cases}
		f(\Gamma), & \mbox{for}\; g_{ijk} \leq 0, \forall (i,j,k) \in \Psi,\\ 
		+\infty, & \mbox{otherwise.} 
	\end{cases}
\end{align*}
This implies that the solution to problem (\ref{optprob}) equals
\begin{equation}\label{infsup}
	\inf_{\Gamma \in \mathcal{C}^d} \sup_{\boldsymbol \eta \geq \mathbf{0}} L(\Gamma,\boldsymbol{\eta}). 
\end{equation}
By duality theory, the Lagrange dual function $ \inf_{\Gamma \in \mathcal{C}^d} L(\Gamma,\boldsymbol{\eta}) $ forms a lower bound for the optimal value of our original problem, and thus the supremum of the Lagrange dual function over $\boldsymbol \eta \geq 0$ is the most tight of those lower bounds. 
This is called weak duality. By the weak Slater's condition, for any convex problem with only affine inequality constraints we obtain strong duality, so that the lower bound mentioned above is a tight bound. 
Therefore, the optimal value for our original/primary problem in \eqref{infsup} is the same as
\begin{equation}\label{supinf}
	\sup_{\boldsymbol \eta \geq \mathbf{0}} \inf_{\Gamma \in \mathcal{C}^d} L(\Gamma,\boldsymbol{\eta}),
\end{equation}
meaning that in the case where strong duality holds, we can switch the infimum and supremum and still obtain the optimal solution of the primary problem.
By Proposition 4.3 in \citet{LZ2022}, the infimum $ \inf_{\Gamma \in \mathcal{C}^d} L(\Gamma,\boldsymbol{\eta}) $ exists and is obtained by the unique $\Gamma \in \mathcal{C}^d$ for which the gradient of (\ref{lagrangeprimal}) equals to the zero matrix (for a fixed $\boldsymbol{\eta}$). 
This is the first KKT condition (stationarity)
\begin{align*}
	\nabla_{\Gamma} L(\Gamma,\boldsymbol{\eta}) &= \widehat{Q} - Q + \sum_{(i,j,k) \in \Psi} \eta_{ijk} \nabla_{\Gamma} g_{ijk}=\mathbf{0},
\end{align*}
so that the infimum with respect to $\Gamma$ is reached where 
\begin{equation}  \label{sollagrangian}
	Q  = \widehat{Q} + \sum_{(i,j,k) \in \Psi} \eta_{ijk} \nabla_{\Gamma} g_{ijk},
\end{equation}
which gives the $Q$ (or equivalently the corresponding $\Gamma$) related to any value of the KKT multiplier $\boldsymbol{\eta}$ which gives a solution to the Lagrange dual function.
We therefore define a linear transform of $ \boldsymbol{\eta} $ to $ \mathbb{S}_0^d $ given $\widehat{Q}$ by
\begin{align}
	Q(\boldsymbol{\eta})=\widehat{Q}+\sum_{(i,j,k)\in \Psi} \eta_{ijk} \nabla_{\Gamma}g_{ijk}, \label{eq:Q_from_eta}
\end{align}
and call the corresponding variogram $ \Gamma(\boldsymbol{\eta}) $, which is the unique minimizer of the Lagrange dual function.
This yields that
\begin{align*}
	\inf_{\Gamma \in \mathcal{C}^d}L(\Gamma,\boldsymbol{\eta})=L(\Gamma(\boldsymbol{\eta}),\boldsymbol{\eta})&= - \log \det(\CM(\Gamma(\boldsymbol{\eta}))) +   \langle Q(\boldsymbol{\eta}), \Gamma(\boldsymbol{\eta})\rangle\\
	&= \log\left(\sum_{T\in \mathcal{T}} \prod_{ij \in T}Q_{ij}(\boldsymbol{\eta})\right) + (d-1),
\end{align*}
where we used $ A(\Gamma)=A(Q) $ in the last equation.
We can now re-express (\ref{supinf}) solely in terms of the KKT multipliers $\boldsymbol{\eta}\in\RR^m$,
so that we obtain the proposition.

\bibliographystyle{chicago} 
\bibliography{bibliography}   
\end{document}